\documentclass{article}
\usepackage[utf8]{inputenc}
\usepackage{graphicx} 

\usepackage{amsfonts}
\usepackage{mathtools}
\usepackage{amsmath, amsthm, amsopn, amssymb, microtype}
\usepackage{mathrsfs} 
\usepackage[all,cmtip]{xy}
\usepackage{enumerate, tikz, etoolbox, intcalc, caption, subcaption, afterpage}
\usepackage[english]{babel}
\usepackage{enumitem}
\usepackage{csquotes}

\usepackage[colorlinks=true,urlcolor=blue,pdfborder={0 0 0}]{hyperref}
\hypersetup{linkcolor=[rgb]{0.8,0,0}}
\hypersetup{citecolor=[rgb]{0,0.6,0}}
\usepackage[nameinlink]{cleveref}

\newtheorem{theorem}{Theorem}[section]
\newtheorem{definition}[theorem]{Definition}
\newtheorem{lemma}[theorem]{Lemma}
\newtheorem{corollary}[theorem]{Corollary}

\newtheorem{remark}[theorem]{Remark}

\usepackage[left=30mm, top=25mm, right=30mm, bottom=25mm]{geometry} 
\title{Regularity of profinite isomorphisms of hyperbolic $3$-manifolds}
\author{Liam Hanany}
\date{\today}

\begin{document}
\maketitle

\begin{abstract}
We strengthen a result of Liu from his papers \cite{liu2023finite}, \cite{liu2025finite}, by proving that profinite isomorphisms of hyperbolic $3$-manifolds are regular in the sense of Boileau and Friedl \cite{boileau2020profinite}.
\end{abstract}

\section{Introduction}
Our main result is the following theorem.

\begin{theorem}
\label{regularity theorem}
Let $M_1, M_2$ be finite-volume hyperbolic $3$-manifolds. Assume that there exists an isomorphism $\Phi: \widehat{\pi_1(M_1)} \to \widehat{\pi_1(M_2)}.$ Then this isomorphism $\Phi$ is regular in the sense of Definition \ref{fibration_regularity}.
\end{theorem}

This improves upon a result of Xu \cite[Theorem 1.4]{xu2025regularity}, who proves the same when the manifolds have non-empty boundary.

Our methods are purely number-theoretic, and the proof of Theorem \ref{regularity theorem} uses Liu's ideas in combination with the following theorem.
We remind the reader that a Laurent polynomial $f \in \mathbb{Z}[t, t^{-1}]$ is said to be reciprocal if $f(t) = t^N \cdot f(t^{-1})$ for some $N \in \mathbb{Z}$.

\begin{theorem}
\label{mu_is_plus_minus_one}
Let $f \in \mathbb{Z}[t, t^{-1}] \simeq \mathbb{Z}[\mathbb{Z}]$ be a monic reciprocal Laurent polynomial. Let $\mu \in \widehat{\mathbb{Z}}^\times$ and assume the equality of continuous ideals $(f(t)) = (f(t^\mu))$ in $\widehat{\mathbb{Z}}[[\widehat{\mathbb{Z}}]].$ Then either $\mu = \pm 1,$ or all roots of $f$ are roots of unity.
\end{theorem}

Here we follow the convention that a root of a Laurent polynomial $f$ in a field $K$ is a non-zero $\alpha \in K$ such that $f(\alpha) = 0$.

\begin{remark}
In \cite[Lemma 6.1]{liu2025finite}, Liu proved that under the same hypotheses as Theorem \ref{mu_is_plus_minus_one}, $\mu^2 = 1$. This still leaves uncountably many options for $\mu$, as at each prime $p$, the projection of $\mu$ to the $p$-adic integers $\mathbb{Z}_p^\times$ is $\pm 1$, but different signs can be chosen at different primes.
\end{remark}

Liu's proof is based on Mahler's result \cite{mahler1935transzendente}. The starting point of our proof will be his result.

Theorem \ref{mu_is_plus_minus_one} is sufficient in proving Theorem \ref{regularity theorem}, as it will be used for Alexander polynomials of fibred hyperbolic $3$-manifolds, which are monic and reciprocal. However, for completeness, we mention that the same result is also true for
general Laurent polynomials.

\begin{theorem}
\label{generalized_mu_is_plus_minus_one}
Let $f \in \mathbb{Z}[t, t^{-1}]$ be a Laurent polynomial. Assume that $(f(t)) = (f(t^\mu))$ as continuous ideals in $\widehat{\mathbb{Z}}[[\widehat{\mathbb{Z}}]].$ Then either $\mu = \pm 1,$ or all roots of $f$ are roots of unity.
\end{theorem}

Ueki \cite[Lemma 3.6]{ueki2018profinite}, building on an earlier result of Fried \cite{fried2006cyclic}, proved that if $f, g \in \mathbb{Z}[t, t^{-1}]$ are reciprocal polynomials and $(f(t)) = (g(t^\mu))$ is an equality of continuous ideals, then $f = g$ up to a unit of $\mathbb{Z}[t,t^{-1}]$ (see also Theorem \ref{Ueki_theorem} below). This result was used in \cite{liu2023finite}. As a corollary of our result and Ueki's, we obtain the following

\begin{corollary}
\label{generalized_Ueki}
Let $f, g \in \mathbb{Z}[t, t^{-1}]$ be two Laurent polynomials. Assume that $(f(t)) = (g(t^{\mu}))$ is an equality of continuous ideals in $\widehat{\mathbb{Z}}[[\widehat{\mathbb{Z}}]].$ Moreover, assume that either $f$ or $g$ have a root that is not a root of unity. Then $\mu = \pm 1,$ and $f(t) = g(t^{\pm 1})$ up to a unit of $\mathbb{Z}[t, t^{-1}].$
\end{corollary}

\subsection{Structure of the note}
In Section $\S$\ref{preliminaries}, we give a brief overview of Liu's results in \cite{liu2023finite}, \cite{liu2025finite}. In Section $\S$\ref{proof_of_regularity} we prove Theorem \ref{mu_is_plus_minus_one} and deduce Theorem \ref{regularity theorem}. In Section $\S$\ref{generalisations} we deal with non-monic non-reciprocal Laurent polynomials, and briefly explain an application to the automorphism group of the profinite completion of a hyperbolic $3$-manifold, as well as a possible application to the question of profinite rigidity of free-by-cyclic groups.

\subsection{Acknowledgements}
The author would like to thank Henry Wilton, Francesco Fournier-Facio and Julian Wykowski for helpful conversations and motivation in pursuing this problem.

\noindent The author was supported by a C T Taylor Cambridge International Scholarship.

\section{Preliminaries}
\label{preliminaries}
In this section we will give a brief reminder on Liu's results from his papers \cite{liu2023finite}, \cite{liu2025finite}.

\begin{remark}
Stallings' Theorem \cite{stallings1962fibering} (see also \cite{hempel19773}) states that for an irreducible $3$-manifold $M$ the kernel of a map $\pi_1(M) \to \mathbb{Z}$ is finitely generated if and only if it is a surface subgroup $\pi_1(S)$. In such a case, the manifold $M$ fibres over the circle $S^1$, with fibres isomorphic to the surface $S$. We will say that such a map $\pi_1(M) \to \mathbb{Z}$ fibres.
\end{remark}

\begin{remark}
Thurston's (semi-)norm is defined on $H^1(M,\mathbb{Z})$, and has norm-ball a (possibly unbounded) convex rational polyhedron. Thurston proved that a cohomology class fibres if and only if it is in one of a subset of the cones over the maximal-dimensional faces of this polyhedron.

Moreover, this semi-norm is in fact a norm when $M$ has no embedded incompressible tori (other than its boundary tori), which happens for example if $M$ is hyperbolic. See \cite{thurston1986norm} for details.
\end{remark}

Throughout this section we let $M_1,M_2$ denote compact orientable finite-volume hyperbolic $3$-manifolds, with empty or toroidal boundary. Assume that there exists a profinite isomorphism $\Phi: \widehat{\pi_1(M_1)} \to \widehat{\pi_1(M_2)}.$ Such an isomorphism also induces a profinite isomorphism $\Phi_\ast: H^1(M_2, \widehat{\mathbb{Z}}) \to H^1(M_1, \widehat{\mathbb{Z}}).$

\begin{definition}
\label{fibration_regularity}
The map $\Phi$ is said to be $\widehat{\mathbb{Z}}^\times$-regular if there is an isomorphism $F: H^1(M_2, \mathbb{Z}) \to H^1(M_1, \mathbb{Z})$ and a unit $\mu \in \widehat{\mathbb{Z}}^\times$ such that $\Phi_\ast = \mu \cdot F.$

The map $\Phi$ is said to be regular if $\mu = \pm 1$.
\end{definition}

The following is \cite[Theorem 5.1, Corollary 5.3, Theorem 6.1, Corollary 6.2, Corollary 6.3]{liu2023finite}
\begin{theorem}
\label{Thurston_norm_alignment}
Any profinite isomorphism of finite-volume hyperbolic $3$-manifolds is $\widehat{\mathbb{Z}}^\times$-regular. Moreover the corresponding mapping $F$ of the discrete homologies of $M_1, M_2$ preserves the Thurston norm, and maps fibred classes to fibred classes.
Further, $\Phi$ induces an isomorphism between the closures in the profinite topology of the fibred surface subgroups.
\end{theorem}

Once we know this theorem, we can align fibrations of the profinitely isomorphic $3$-manifolds, as described in the following commutative diagram:

\begin{equation}
\label{aligned_fibrations}
\begin{gathered}
\xymatrix{
    1 \ar[r]
    & \widehat{\pi_1(S_1)} \ar[r] \ar[d]^{\mathrel{\rotatebox[origin=c]{-90}{$\simeq$}}}
    & \widehat{\pi_1(M_1)} \ar[r] \ar[d]^{\Phi}
    & \widehat{\mathbb{Z}} \ar[r] \ar[d]^{\cdot \mu}
    & 1
    \\ 1 \ar[r]
    & \widehat{\pi_1(S_2)} \ar[r]
    & \widehat{\pi_1(M_2)} \ar[r]
    & \widehat{\mathbb{Z}} \ar[r]
    & 1
}
\end{gathered}
\end{equation}

In this case we will say that the two fibrations of $M_1, M_2$ are \emph{aligned} with respect to the profinite isomorphism $\Phi$. The monodromies of these fibrations are a pair of mapping classes $f_1,f_2$ of the isomorphic surfaces $S_1, S_2$. Since the $3$-manifolds $M_1,M_2$ are hyperbolic, these mapping classes are pseudo-Anosov.

The following is \cite[Lemma 7.1]{liu2025finite}
\begin{theorem}
\label{rigidity_mu_squared_is_one}
The $\widehat{\mathbb{Z}}^\times$-regularity constant $\mu$ satisfies $\mu^2 = 1$.
\end{theorem}

The proof of this theorem is a combination of four results, which we describe below as they will be used in the proof of Theorem \ref{regularity theorem}.

The following is \cite[Lemma 7.8]{liu2023finite}
\begin{theorem}
\label{equality_of_twisted_alexander_polynomials}
Given aligned fibrations of a profinite isomorphism $\Phi$ of hyperbolic $3$-manifolds, the (twisted) Alexander polynomials of $f_1, f_2$ satisfy the equality of continuous principal ideals in $\widehat{\mathbb{Z}}[[\widehat{\mathbb{Z}}]]$
$$
(\Delta_{f_1}(t^\mu)) = (\Delta_{f_2}(t)).
$$
\end{theorem}

Liu then uses the following result due to Ueki \cite[Lemma 3.6]{ueki2018profinite}.
\begin{theorem}
\label{Ueki_theorem}
Let $f, g \in \mathbb{Z}[t, t^{-1}] \simeq \mathbb{Z}[\mathbb{Z}]$ be reciprocal Laurent polynomials. Assume that there is a unit $\mu \in \widehat{\mathbb{Z}}^\times$ such that the continuous principal ideals $(f(t^{\mu})) = (g(t))$ in $\widehat{\mathbb{Z}}[[\widehat{\mathbb{Z}}]]$ are equal. Then $f, g \in \mathbb{Z}[\mathbb{Z}]$ are equal up to units, i.e. the ideals in $\mathbb{Z}[\mathbb{Z}]$ generated by $f,g$ are equal.
\end{theorem}

The following result is also due to Liu \cite{liu2020virtual} and states that pseudo-Anosov mapping classes virtually have eigenvalues outside the unit circle. We state this in the following theorem

\begin{theorem}
\label{virtual_homological_eigenvalues}
Any fibred hyperbolic $3$-manifold $M$ has a finite cover with Alexander polynomial having a root outside the unit circle.
\end{theorem}

The final ingredient in proving Theorem \ref{rigidity_mu_squared_is_one} is \cite[Lemma 6.1]{liu2025finite}:
\begin{theorem}
\label{mu_squared_is_one}
Let $f \in \mathbb{Z}[\mathbb{Z}] \simeq \mathbb{Z}[t, t^{-1}]$ and $\mu \in \widehat{\mathbb{Z}}^\times$. Assume that the continuous principal ideals $(f(t)) = (f(t^\mu))$ are equal in $\widehat{\mathbb{Z}}[[\widehat{\mathbb{Z}}]].$ Then $\mu^2 = 1$, or all roots of $f$ are roots of unity.
\end{theorem}

\section{Proof of regularity}
\label{proof_of_regularity}
Our proof will use the following result due to Schinzel \cite{schinzel1974primitive}.
\begin{theorem}
\label{Schinzel's Theorem}
Let $\alpha$ be an algebraic number in a number field $K$. Assume that $\alpha$ is not a root of unity. Then for any sufficiently large $n$ there exists a prime ideal $\mathfrak{r}$ of $K$ such that the order of $\alpha$ in $(\mathcal{O}_K / \mathfrak{r})^\times$ is $n.$
\end{theorem}
\begin{remark}
Equivalently stated, whenever an algebraic number is not multiplicative torsion, then above a certain bound, it can have arbitrary orders in the multiplicative groups of the residue fields at primes. This is reminiscent of the notion of potence (and omnipotence) in finitely generated groups, which has already been used in proving various forms of profinite rigidity, see for instance \cite{wilton2024congruence}.
\end{remark}

We are now ready to prove Theorem \ref{mu_is_plus_minus_one}. In the proof we will not use the fact that the Laurent polynomial $f$ is monic or reciprocal, rather only the conclusion of Liu's theorem that $\mu^2 = 1.$ This will be used in the next section.

\begin{proof}[Proof of Theorem \ref{mu_is_plus_minus_one}]
Assume that $f$ has a root $\alpha$ that is not a root of unity. Then by Theorem \ref{mu_squared_is_one} we know that $\mu^2 = 1.$ Therefore for any prime $p \in \mathbb{Z}$, the projection of $\mu$ to $\mathbb{Z}_p^\times$ is $\mu_p = \pm 1.$

Assume by contradiction that there are primes $p, q \in \mathbb{Z}$ such that $\mu_p = 1$ and $\mu_q = -1$. Assume at first (for simplicity) that both $p, q \neq 2$.

Let $K$ be a number field containing all roots of the polynomial $f$. Let $N$ be an upper bound on all possible orders of roots of unity appearing among $\beta \cdot \gamma, \beta / \gamma$ where $\beta, \gamma$ are roots of $f$.

Pick a pair of powers $p^n, q^m$ of $p, q$ both larger than our bound $N$, as well as the bound occuring in Schinzel's theorem for $\alpha$.
By Theorem \ref{Schinzel's Theorem} we have a prime ideal $\mathfrak{r}$ of $K$ over some rational prime $r \in \mathbb{Z}$ such that the order of $\alpha$ in $\mathcal{O}_K/\mathfrak{r}$ is $s \coloneq p^n \cdot q^m$.

Consider $\alpha$ as an element of the $\mathfrak{r}$-adic completion $(\mathcal{O}_K)_\mathfrak{r}$. Then $\alpha$ is a unit in this ring as it is invertible mod $\mathfrak{r}$. The multiplicative group of this ring $(\mathcal{O}_K)_\mathfrak{r}^\times$ is a profinite group, and so we may consider $\beta = \alpha^\mu \in (\mathcal{O}_K)_\mathfrak{r}^\times$. Since $\beta$ is also a root of $f$ by our assumption, it follows that $\beta$ is an algebraic number, and so $\beta \in K.$

Since $\alpha^s = 1 (\mathrm{mod}\ \mathfrak{r})$, the element $\alpha^s$ belongs to $\mathrm{ker}((\mathcal{O}_K)_\mathfrak{r}^\times \to (\mathcal{O}_K/\mathfrak{r})^\times)$ which is a pro-$r$ group (see for instance \cite[Chapter II, Proposition 3.10]{neukirch1992algebraic}). Therefore, $\beta^s = (\alpha^s)^\mu = (\alpha^s)^{\mu_r} = (\alpha^s)^{\pm 1}$. In particular $(\beta \cdot \alpha^{\pm 1})^s = 1$.
However, since $\alpha, \beta \in K$ are roots of $f$, by assumption $\beta \cdot \alpha^{\pm 1}$ is a root of unity of order $\leq N$. On the other hand, considering $\beta \cdot \alpha^{\pm 1}$ mod $\mathfrak{r}$, we get $\alpha^{\mu \pm 1}$. Since $\alpha$ has order $s = p^n \cdot q^m$ and $\mu= 1 (\mathrm{mod}\ p^n), \mu = -1(\mathrm{mod}\ q^m)$, $\alpha^{\mu \pm 1} \in (\mathcal{O}_K/\mathfrak{r})^\times$ has order that is either $p^n$ or $q^m$. But this is in contradiction as both are greater than $N$.

In order to deal with the case that either $p$ or $q$ is $2$, simply choose the powers $p^n, q^m$ to be strictly greater that $2N$, and then note that the order of $\alpha^{\mu \pm 1}$ is either $\frac{p^n}{\gcd(p^n, 2)}$ or $\frac{q^m}{\gcd(q^m, 2)}$ and therefore still greater than $N.$
\end{proof}

\begin{proof}[Proof of Theorem \ref{regularity theorem}]
By Theorem \ref{Thurston_norm_alignment}, if one of the $3$-manifolds is fibred then so is the other one (this was proved in greater generality by \cite{jaikin2020recognition}).

In general, using the Agol--Wise virtual fibring \cite{agol2008criteria}, \cite{agol2013virtual}, \cite{wise2021structure}, we can find a finite cover of $M_1$ that is fibred, and consider the restriction of $\Phi$ to this finite index subgroup to get a corresponding finite fibred cover of $M_2$. We can then align two fibrations as in Diagram \ref{aligned_fibrations}.

By Theorem \ref{virtual_homological_eigenvalues}, we find a finite fibred cover $\overline{M_1}$ of $M_1$ with Alexander polynomial having roots outside the unit circle. Consider the corresponding cover $\overline{M_2}$ of $M_2$. By Theorem \ref{equality_of_twisted_alexander_polynomials}, we get equality of Alexander polynomials $(\Delta_{\overline{M_1}}(t^\mu)) = (\Delta_{\overline{M_2}}(t))$. By Ueki's result, $\Delta_{\overline{M_1}} = \Delta_{\overline{M_2}}$. By Theorem \ref{mu_is_plus_minus_one}, $\mu = \pm 1.$
\end{proof}

\section{Generalisations and Applications}
\label{generalisations}
\subsection{Non-monic non-reciprocal polynomials}
The first step will be to prove an analogue of Theorem \ref{mu_squared_is_one} without the assumption that the Laurent polynomial $f$ is monic and reciprocal. The proof uses the same ideas as the proof in \cite{liu2025finite}, with some slight modifications. We repeat the proof here with the appropriate modification for completeness.
\begin{lemma}
\label{generalised_mu_squared_is_one}
Let $f \in \mathbb{Z}[t, t^{-1}]$ be a Laurent polynomial. Let $\mu \in \widehat{\mathbb{Z}}^\times$ be a unit. Assume that $(f(t)) = f(t^\mu))$ is an equality of continuous ideals in $\widehat{\mathbb{Z}}[[\widehat{\mathbb{Z}}]].$ Moreover, assume that $f$ has a root that is not a root of unity. Then for every sufficiently large prime $p$ in $\mathbb{Z}$, the projection of $\mu$ to $\mathbb{Z}_p$ is $\pm 1.$
\end{lemma}
\begin{proof}
Let $\alpha$ be a root of $f$ that is not a root of unity. Consider a number field $K$ containing $\alpha$.

Let $p$ be a prime. Choose a prime ideal $\mathfrak{p}$ of $K$ over $p$. Let $(O_K)_\mathfrak{p}$ be the completion of the ring of integers of $K$ at the prime ideal $\mathfrak{p}.$ Since $\alpha$ has non-zero valuation only at finitely many primes of $K$, for sufficiently large $p$, we have that $\alpha \in (\mathcal{O}_K)_\mathfrak{p}^\times$.

The remainder of the proof now follows \cite[Lemma 6.1]{liu2025finite} precisely. We do not claim any novelty.

By the ideal equality, we know that $\alpha \in (\mathcal{O}_K)_\mathfrak{p}^\times$ is a root of $f$ implies that $\alpha^\mu \in (\mathcal{O}_K)_\mathfrak{p}^\times$ is also a root of $f.$ Therefore $\alpha^{\mu^n}$ is a root of $f$ for every $n$. As $f$ has finitely many roots, $\alpha^{\mu^d} = \alpha$ for some $d$.

In particular, $\alpha^{m \cdot \mu^d} = \alpha^m$ for every integer $m$. Choosing $m$ to be divisible by the order of $(\mathcal{O}_K/\mathfrak{p})^\times$, we may assume that $\alpha^m \in \ker((\mathcal{O}_K)_\mathfrak{p}^\times \to (\mathcal{O}_K/\mathfrak{p})^\times)$ which is a pro-$p$ group. Therefore $\alpha^m = \alpha^{m \cdot \mu^d} = \alpha^{m \cdot \mu_p^d},$ for the projection $\mu_p$ of $\mu$ to $\mathbb{Z}_p.$

Since $\alpha$ is not a root of unity, it now follows that $\mu_p^d = 1.$ Indeed, assume by contradiction that $\mu_p^d - 1 = p^r \cdot \nu$ for some $\nu \in \mathbb{Z}_p^\times.$ Taking a sufficiently large power $\xi = \alpha^m$, we may assume that $\xi - 1$ has arbitrarily large valuation at $\mathfrak{p}$. However, as $\xi^{p^r \cdot \nu} = 1,$ the $\mathfrak{p}$-valuation of $\xi - 1$ is bounded -- in contradiction. (See \cite[Lemma 6.1, Remark 6.2]{liu2025finite} for further details.)

Now, let $\beta = \alpha^\mu$. Since $\beta$ is a root of $f$, it is an algebraic number. Taking a sufficiently large power of $\alpha, \beta$ we may assume that $|\alpha - 1|_p, |\beta - 1|_p < \frac{1}{p}$. Then, by Mahler's theorem \cite{mahler1935transzendente}, $\mu_p = \frac{\log_p(\beta)}{\log_p(\alpha)}$ is either rational or transcendental. Since $\mu_p^d = 1$ it is algebraic, hence rational. Since $\mu_p$ is a rational root of unity, it is $\pm 1.$
\end{proof}

We can now repeat the proof of Theorem \ref{mu_is_plus_minus_one}.
\begin{proof}[Proof of Theorem \ref{generalized_mu_is_plus_minus_one}]
Let $\alpha$ be a root of $f$ that is not a root of unity. Let $K$ be a number field containing all roots of $f$.

Let $M$ be a natural number. We will prove that $\mu \equiv \pm 1 \ (\mathrm{mod}\ M).$ Since this will be true for an arbitrary number $M$, it follows that $\mu = \pm 1$.

As before, consider all pairs of roots $\beta, \gamma$ of $f$, and let $N$ be a number divisible by all orders of roots of unity appearing among $\beta \cdot \gamma, \frac{\beta}{\gamma}$, and greater than the bound appearing in Schinzel's theorem for $\alpha$.

By Schinzel's theorem, let $\mathfrak{r}$ be a prime such that $\alpha$ has order $s = N \cdot M$ in $(\mathcal{O}_K/\mathfrak{r})^\times.$ We further enlarge $N$ such that this prime $r$ satisfies the conclusion of Lemma \ref{generalised_mu_squared_is_one}.

As before, $\beta = \alpha^\mu$ is also a root of $f$. Then $\alpha^s, \beta^s$ are both elements of $\ker((\mathcal{O}_K)_\mathfrak{r}^\times \to (\mathcal{O}_K/\mathfrak{r})^\times)$ which is a pro-$r$ group. Therefore $\beta^s = (\alpha^s)^\mu = (\alpha^s)^{\mu_r} = \alpha^{\pm s}.$
Hence, $(\beta \cdot \alpha^{\pm 1})^s = 1,$ so $\beta \cdot \alpha^{\pm 1}$ is a root of unity.

By the definition of $N$, $\alpha^{N(\mu \pm 1)} = (\beta \cdot \alpha^{\pm 1})^N = 1$ in $(\mathcal{O}_K/\mathfrak{r})^\times.$ As $\alpha$ has order $s = N \cdot M$, it follows that $\mu \pm 1 \equiv 0\ (\mathrm{mod}\ M).$
\end{proof}

In the following proof, for any field $K$ we let $\overline{K}$ denote an algebraic closure of $K.$
\begin{proof}[Proof of Corollary \ref{generalized_Ueki}]
Applying the automorphism $t \mapsto t^{-1}$, we get the equality of ideals $(f(t^{-1})) = (g(t^{-\mu})).$ Therefore $(f(t)f(t^{-1})) = (g(t^\mu)g(t^{-\mu}))$. Hence the reciprocal polynomials $f(t)f(t^{-1}),g(t)g(t^{-1})$ satisfy the conditions of Ueki's theorem (Theorem \ref{Ueki_theorem}) and therefore $f(t)f(t^{-1}) = g(t)g(t^{-1}).$ By Theorem \ref{generalized_mu_is_plus_minus_one}, $\mu = \pm 1.$ Replacing $g(t)$ with $g(t^{-1})$ if necessary, we assume that $\mu = 1.$

Therefore $(f(t)) = (g(t))$ as continuous ideals. Choose a sufficiently large prime $p$ such that every root of $f$ has trivial valuation at every prime over $p$.

Hence, every root $\alpha$ of $f$ in $\overline{\mathbb{Q}_p}$ is also a root of $g$. So, $f(t), g(t)$ have the same roots (and multiplicities) in $\overline{\mathbb{Q}_p}$, and therefore the same is true in $\overline{\mathbb{Q}}$. So, $f(t) = c \cdot g(t) \cdot t^N$, for some $c \in \mathbb{Q}.$ Multiplying by an appropriate power of $t$, we get that $f(t) = cg(t)$. Hence $f(t) = a\cdot h(t), g(t) = b\cdot h(t)$ for a primitive $h$ and $a, b \in \mathbb{Z}$. Considering the image of these ideals in $(\mathbb{Z}/a)[[\widehat{\mathbb{Z}}]]$, we see that $b$ is divisible by $a$. By symmetry, $a = \pm b$, as required.
\end{proof}

\begin{remark}
We mention that an ideal equality of the form appearing in Corollary \ref{generalized_Ueki} also appeared in Wykowski's work on absolute profinite rigidity of solvable Baumslag-Solitar groups. In Equation 5.8 in the proof of \cite[Theorem A]{wykowski2025profinite}, the equality given is equivalent to the equality of continuous ideals $(a \cdot t - b) = (t^\kappa - n),$ where $a, b, n \in \mathbb{Z}$ and $\kappa \in \widehat{\mathbb{Z}}^\times.$ Wykowski used a resolution to a problem of Erdős \cite[Theorem 1]{schoof1997support} to deduce $n = (\frac{b}{a})^{\pm 1}.$ This also follows from Corollary \ref{generalized_Ueki}.
\end{remark}

\subsection{Profinite automorphisms of hyperbolic 3-manifolds}
Another corollary of our results is the following
\begin{corollary}
    Let $M$ be a finite-volume hyperbolic $3$-manifold. Then the image of the natural map $$\mathrm{Aut}(\widehat{\pi_1(M)}) \to \mathrm{Aut}(H^1(M; \widehat{\mathbb{Z}}))$$
    is finite.
\end{corollary}
\begin{proof}
    Denote $n = \mathrm{dim}(H^1(M, \mathbb{Z})),$ and let $G \leq \mathrm{Aut}(H^1(M; \widehat{\mathbb{Z}})) \simeq \mathrm{GL}_n(\widehat{\mathbb{Z}})$ denote the image of $\mathrm{Aut}(\widehat{\pi_1(M)})$ under the natural homomorphism. By Theorem \ref{regularity theorem}, $G \leq \mathrm{Aut}(H^1(M, \mathbb{Z})) \simeq \mathrm{GL}_n(\mathbb{Z}).$ Since $\mathrm{GL}_n(\mathbb{Z})$ has a torsion-free finite-index subgroup, it follows that a finite-index subgroup of $G$ is torsion-free. Hence, it suffices to show that the image of every element $\Phi \in \mathrm{Aut}(\widehat{\pi_1(M)})$ in $G$ has finite order. Denote by $A$ the image of $\Phi$ in $G$.

    For every $\mu \in \widehat{\mathbb{Z}}^\times$, the automorphism $\Phi^\mu \in \mathrm{Aut}(\widehat{\pi_1(M)})$ has image $A^\mu \in \mathrm{GL}_n(\widehat{\mathbb{Z}})$ which is contained in $G \leq \mathrm{GL}_n(\mathbb{Z}).$

    Assume by contradiction that $A$ has an eigenvalue $\lambda$ which is not a root of unity. Fix an arbitrary $\mu \in \widehat{\mathbb{Z}}^\times,\ \mu \neq \pm 1$ and consider $B = A^\mu \in \mathrm{GL}_n(\mathbb{Z}).$ Let $P(t), Q(t) \in \mathbb{Z}[t] \simeq \mathbb{Z}[\mathbb{Z}]$ denote the characteristic polynomials of $A, B$ respectively. Then, in the ring $\widehat{\mathbb{Z}}[[\widehat{\mathbb{Z}}]]$, there is an equality of continuous ideals $(P(t)) = (Q(t^\mu))$ (see for instance \cite[Lemma 7.5]{liu2023finite}). Since $P$ has a root that is not a root of unity, $\mu = \pm 1$ by Corollary \ref{generalized_Ueki}, in contradiction.

    Since all eigenvalues of $A$ are roots of unity, $A$ has a finite power $U = A^n$ which is unipotent. Then every power $U^\mu, \mu \in \widehat{\mathbb{Z}}^\times$ has integer coefficients. We will prove that this implies $U$ has integer coefficients.

    The matrix $U$ is conjugate in $\mathrm{GL}_n(\mathbb{Q})$ to a Jordan normal form with all eigenvalues equal to $1.$ We assume by contradiction that $U \neq 1$, i.e. $U$ has at least one Jordan block of size at least $2.$ It follows that there are a pair of linearly independent vectors $u, v \in \mathbb{Q}^n$ such that $U \cdot u = u,\ U \cdot v = v + u.$
    
    Denote $\Lambda = \mathbb{Z}^n \leq \mathbb{Q}^n$ the natural $U$-preserved lattice. Denote $\Lambda_u = \mathrm{span}_\mathbb{Q}(u) \cap \Lambda, \Lambda_{u, v} = \mathrm{span}_\mathbb{Q}(u, v) \cap \Lambda.$ Then the sub-lattices $\Lambda_u \leq \Lambda_{u,v} \leq \mathbb{Z}^n$ are both preserved by the matrix $U.$ The quotient $\Lambda_{u, v}/\Lambda_u$ is torsion-free. So, we may renormalise the vector $u$ by a rational multiple, and pick a complementary basis element in $\Lambda_{u,v}$, obtaining a basis $\tilde{u},\tilde{v}$ of $\Lambda_{u,v}$. These satisfy that $U \cdot \tilde{u} = \tilde{u}, U \cdot \tilde{v} = \tilde{v} + k \cdot \tilde{u},$ for some $0 \neq k \in \mathbb{Z}.$

    It now follows that $U^\mu \cdot \tilde{v} = \tilde{v} + k \cdot \mu \cdot \tilde{u}$, which is not an element of $\Lambda$ when $\mu \in \widehat{\mathbb{Z}},\ \mu \neq \pm 1.$
\end{proof}

\begin{remark}
    Corollary \ref{generalized_Ueki} was used in the proof for simplicity, but it was unnecessary. Instead, we could consider an arbitrary eigenvalue $\lambda$ which is not a root of unity, and consider its value in a $p$-adic completion of the corresponding ring of integers of the number field $\mathbb{Q}(\lambda).$ We can then use the fact that not every $\mathbb{Z}_p^\times$-power of $\tilde{\lambda}$ in this $p$-adic completion is an algebraic integer, to deduce that a corresponding power of $A$ is not an integer matrix.
\end{remark}

We note that it also follows that for every characteristic cover $\overline{M} \to M$, the image of the map $\mathrm{Aut}(\widehat{\pi_1(M)}) \to \mathrm{Aut}(H^1(\overline{M}, \widehat{\mathbb{Z}}))$ is finite. However - it does not follow from the proof that there is a uniform upper bound on the size of the images of these groups over all characteristic covers $\overline{M} \to M.$

\subsection{Profinite rigidity of free-by-cyclic groups}
Profinite rigidity of free-by-cyclic groups was studied by Bridson--Reid \cite{bridson2020profinite} and by Hughes--Kudlinska \cite{hughes2025profinite}. An analogue of Theorem \ref{equality_of_twisted_alexander_polynomials} for free-by-cyclic groups was proved in \cite[Lemma 5.8]{hughes2025profinite}. Hadari \cite[Theorem 1.5]{hadari2020homological} proved the analogue of Theorem \ref{virtual_homological_eigenvalues} for free-by-cyclic groups with fully irreducible monodromy. Using these and repeating the proof of Theorem \ref{regularity theorem}, we obtain the following corollary of our results
\begin{corollary}
\label{free-by-cyclic-regularity}
Assume that every profinite isomorphism of free-by-cyclic groups with fully irreducible monodromy is $\widehat{\mathbb{Z}}^\times$-regular. Then every profinite isomorphism of free-by-cyclic groups with fully irreducible monodromy is regular.
\end{corollary}
Therefore, given an analogue of Theorem \ref{Thurston_norm_alignment} for free-by-cyclic groups, see also \cite[Question 1.1]{hughes2025profinite}, any profinite isomorphism of fully irreducible free-by-cyclic groups would be regular in the sense of Definition \ref{fibration_regularity}.

\begin{remark}
Any automorphism $\varphi \in \mathrm{Out}(F_n)$ which is atoroidal exponentially growing will also satisfy the conclusion of Hadari's Theorem. Therefore the statement of Corollary \ref{free-by-cyclic-regularity} also holds if we replace the fully irreducible assumption by the assumption the monodromy is atoriodal exponentially growing in both the hypothesis and the conclusion.

To see this -- we first note that we may freely pass between $\varphi$ and its powers, as $\varphi$ has a virtual homological eigenvalue outside the unit circle if any only some positive power of $\varphi$ has such an eigenvalue. Hence, if $\varphi$ is not fully irreducible, we may replace $\varphi$ by a reducible power of $\varphi$, i.e. that has a preserved free-factor $H \overset{\ast}{\leq} F_n$. If the sub-action $\varphi|_H$ is fully irreducible, we may use Hadari's result to find a virtual homological eigenvalue off the unit circle. Otherwise, possibly after taking a power of $\varphi$, there is a smaller preserved free factor. Continuing in this way -- we may assume that $\mathrm{rk}(H) = 1.$ But this contradicts the fact that $\varphi$ is atoroidal.
\end{remark}


\footnotesize
\bibliographystyle{alpha}
\bibliography{bibli}
\vspace{0.5cm}

\vspace{0.5cm}

\noindent{\textsc{DPMMS, Centre for Mathematical Sciences, Wilberforce Road, Cambridge, CB3 0WB, UK}}

\noindent{\textit{Email address:} \texttt{lh917@cam.ac.uk}}

\end{document}